\documentclass[final,10pt]{amsart}
\usepackage{amsmath,amsfonts,amssymb,amscd,verbatim}
\usepackage{dsfont}
\usepackage[margin=0.4in]{geometry}
\usepackage{fullpage}
\usepackage{color}

\usepackage{enumerate}

\usepackage{bm}

\makeatletter
\newcommand{\leqnomode}{\tagsleft@true}
\newcommand{\reqnomode}{\tagsleft@false}
\makeatother

\numberwithin{equation}{section}
\theoremstyle{plain}
\newtheorem{thm}{Theorem} 
\newtheorem{theorem}[equation]{Theorem} 

\newtheorem{corollary}[equation]{Corollary} 
\newtheorem{lemma}[equation]{Lemma}
\newtheorem{proposition}[equation]{Proposition}
\newtheorem{remark}[equation]{Remark}

\theoremstyle{definition}

\newtheorem{example}[equation]{Example}
\newtheorem{hypothesis}[equation]{Hypothesis} 
\newtheorem{question}[equation]{Question}

\DeclareMathOperator\rAut{rAut}
\DeclareMathOperator\Aut{Aut}

\DeclareMathOperator\gr{gr}

\DeclareMathOperator\lm{lm}

\DeclareMathOperator\reg{reg}
\DeclareMathOperator\Span{Span}
\DeclareMathOperator\tr{tr}

\newcommand\NN{\mathbb N}

\newcommand\ZZ{\mathbb Z}

\newcommand\cB{\mathcal B}

\newcommand\fp{\mathfrak{p}}

\newcommand\cnt{\mathcal Z}
\renewcommand{\int}{\mathrm{int}}
\newcommand\inv{^{-1}}

\newcommand{\Maxspec}{\ensuremath{\operatorname{Maxspec}}}
\newcommand\iso{\cong}
\newcommand\kk{\mathds{k}}
\newcommand\kq{\kk[q^{\pm 1}]}
\newcommand\tensor{\otimes}

\newcommand\qp{\kk_\mu[u,v]}
\newcommand\mq{M_\mu}
\newcommand\wa{A_1^\mu}

\newcommand\Aep{A_{\epsilon}}
\newcommand\Bep{B_{\epsilon}}
\newcommand\Cep{C_{\epsilon}}
\newcommand\Rep{R_{\epsilon}}
\newcommand\tornado{\xi}

\newcommand{\qbinom}[2]{\genfrac{[}{]}{0pt}{0}{#1}{#2}}

\renewcommand\mod{~\mathrm{mod}~}

\title{Discriminants of Taft algebra smash products and applications}

\author[Gaddis]{Jason Gaddis}
\address{Miami University, Department of Mathematics, 301 S. Patterson Ave., Oxford, Ohio 45056} 
\email{gaddisj@miamioh.edu}

\author[Won]{Robert Won}
\address{Wake Forest University, Department of Mathematics and Statistics, P. O. Box 7388, Winston-Salem, NC 27109} 
\email{wonrj@wfu.edu}

\author[Yee]{Daniel Yee}
\address{Bradley University, Department of Mathematics, 1501 W. Bradley Ave., Peoria, IL 61625}
\email{dyee@fsmail.bradley.edu}

\begin{document}

\begin{abstract} 
A general criterion is given for when the center of a Taft algebra smash product is the fixed ring. 
This is applied to the study of the noncommutative discriminant. 
Our method relies on the Poisson methods of Nguyen, Trampel, and Yakimov,
but also makes use of Poisson Ore extensions.
Specifically, we fully determine the inner faithful actions of Taft algebras on quantum planes and quantum Weyl algebras. 
We compute the discriminant of the corresponding smash product and apply it to 
compute the Azumaya locus and restricted automorphism group.
\end{abstract}

\maketitle

\section{Introduction}

Throughout $\kk$ is 
an algebraically closed, characteristic zero field
and all algebras are $\kk$-algebras.
Given an algebra $R$, we denote its set of units by $R^\times$ 
and its center by $\cnt(R)$.
For $a,b \in R$ we write $a =_{R^\times} b$ if there exists $c \in R^\times$
such that $a=cb$.

The discriminant is an important invariant of an algebra
and has been adapted to the noncommutative setting by
Ceken, Palmieri, Wang, and Zhang.
It has been used to study automorphism and isomorphism problems, 
locally nilpotent derivations, and more recently the Azumaya loci of PI algebras \cite{BY,CPWZ1,CPWZ2,NTY}.

In \cite{GKM}, Kirkman, Moore, and the first-named author
gave a formula for computing the discriminant of certain skew group algebras
and applied this to compute automorphism groups.
The goal of this paper is to consider this problem for certain smash products
by $H_n(\lambda)$, the $n$th Taft algebra.
Such actions have been studied previously \cite{All,B,KW}.

Given an algebra $R$ and a Hopf algebra $H$,
we say that $H$ {\sf acts on $R$} (from the left)
if $R$ is a left $H$-module via $h\otimes r\mapsto h\cdot r$,
$h \cdot 1_R = \varepsilon(h)1_R$, and
$h \cdot (rr') = \sum (h_1 \cdot r)(h_2 \cdot r')$
for all $h \in H$ and $r,r' \in R$.
Alternatively, we say $R$ is an {\sf $H$-module algebra}.
The action is said to be {\sf inner faithful} if there is 
no nonzero Hopf ideal that annihilates $R$.
When $R$ is a left $H$-module algebra, the 
{\sf smash product algebra} $R\# H$ is
$R \tensor H$ as a $
\kk$-vector space, with elements denoted by $r\#h$ for $r \in R$ and $h \in H$, 
and multiplication given by
\[ (r\#h)(r'\#h') = \sum a(h_1 \cdot b)\#h_2h' 
\quad\text{for all } r,r' \in R \text{ and } h,h' \in H.\]

Let $n > 1$ and $\lambda$ be a primitive $n$th root of unity.
The $n$th Taft algebra $H_n(\lambda)$ \cite{T} is the $\kk$-algebra generated
by $g$ and $x$ subject to the relations
\[ g^n=1, \quad x^n=0, \quad xg = \lambda gx.\]
The coalgebra structure on $H_n(\lambda)$ is given by
\begin{align*}
&\Delta(g)=g \tensor g, \quad \Delta(x) = g \tensor x + x \tensor 1, \qquad
\varepsilon(g) = 1, \quad \varepsilon(x)=0,
\end{align*}
and the antipode by $S(g) = g^{n-1}$, $S(x) = -g^{n-1}x$.
This gives $H_n(\lambda)$ the structure of a Hopf algebra.
For an $H_n(\lambda)$-module algebra $R$, we set  
$R^{\langle x \rangle} = \{ r \in R \mid x(r) = 0\}$ and 
$R^{\langle g \rangle} = \{ r \in R \mid g(r) = r\}$.
Furthermore, we set 
$R^{H_n(\lambda)} 
	= \{ r \in R \mid h(r) 
	= \varepsilon(h)r \mbox{ for all } h \in H_n(\lambda) \}$
to be the {\sf fixed ring} of $R$ under the $H_n(\lambda)$-action.
It is not difficult to see that 
$R^{H_n(\lambda)} = R^{\langle x \rangle}  \cap R^{\langle g \rangle}$.

Recall that an {\sf inner automorphism} of an algebra $R$ is one that is given by conjugation, i.e. $r\mapsto uru^{-1}$, where $u\in R^\times$.
We say an automorphism $g$ of a domain $R$ is {\sf $X$-inner} if
there exists $a \in R$ such that $ra=ag(r)$ for all $r \in R$.
This is not the standard definition of $X$-inner but agrees
when $R$ is a prime Goldie algebra because in this case the automorphism
$g$ becomes inner when extended to the classical quotient ring \cite[Theorem 1.4]{M}.
%For a prime algebra $R$, this is equivalent to extending to the classical quotient ring.

\begin{thm}[Theorem \ref{thm.cnt}]
Let $H=H_n(\lambda)$ and let $R$ be an $H$-module algebra that is a domain.
Suppose that the action of $H$ on $R$ is inner faithful and that no nontrivial power of $g$ 
is $X$-inner when restricted to $R^{\langle x \rangle}$.
Then $\cnt(R\# H) = \cnt(R) \cap R^H$.
\end{thm}

Specifically, we focus on inner faithful actions of $H=H_n(\lambda)$ 
on the quantum plane $A = \qp$ or on the quantum Weyl algebra $A = \wa$ such that $A\#H$ is prime.
By the previous theorem, we obtain $\cnt(A\#H) = A^H = \kk[u^n, v^n]$ when $|\mu| = n > 1$
(Corollary \ref{cor.qpcnt}).

For algebras that may be realized as a specialization, 
Brown and Gordon \cite{BG} showed that there is an induced Poisson structure on the center.
The techniques of Nguyen, Trampel, and Yakimov \cite{NTY}
allow one to determine the factors of the discriminants of such algebras
by first finding the Poisson prime elements of the center.
We realize $A\#H$ as a quotient of an Ore extension, 
which itself may be realized as a specialization.
In contrast to previous work, wherein it was required to have
prior knowledge of the Poisson geometry related to the induced Poisson structure,
we show that there is a connection between Ore extensions and Poisson Ore extensions
(Proposition \ref{prop.indp2}).
We then apply the methods of Oh \cite{OH1,OH2} to find the Poisson primes.

\begin{thm}[Theorem \ref{thm.disc}] 
Let $H=H_n(\lambda)$ and $A = \qp$ or $A=\wa$ with $|\mu|=n>1$.
If $H$ acts linearly and inner faithfully on $A$, then
$d(A\#H /A^H) =_{\kk^\times} u^{2n^4(n-1)}$.
\end{thm}

We apply the discriminant to determine the Azumaya locus of $A\#H$
(Corollary \ref{cor.azumaya}).
Additionally, when when $n=2$, we determine the subgroup of $\Aut(A\#H)$ 
that fixes $H$ up to scalar, which we call the $H$-restricted automorphism group of $A\#H$, denoted $\rAut(A\#H)$
(Theorem \ref{thm.raut}).

\section{Taft actions on quantum algebras}
\label{sec.actions}

Two well-known families of quantum algebras are the  quantum planes 
\[ \qp = \kk\langle u,v \mid uv-\mu vu\rangle\] and 
the quantum Weyl algebras 
\[ \wa = \kk\langle u,v \mid uv-\mu vu-1\rangle.\]
Both have a $\kk$-basis $\{u^iv^j\}$ and this defines a natural
filtration by degree.
Moreover, both algebras are generated in degree one.
We study inner faithful actions of Taft algebras on 
$A=\qp$ or $A=\wa$ with the added hypothesis that the action is {\sf linear}, 
that is, $g(u),g(v),x(u),x(v) \in \kk u + \kk v$.

Recall that when $\mu \neq \pm 1$ we have
$\Aut(\qp)=(\kk^\times)^2$ and
when $\mu = -1$,
$\Aut(\kk_{-1}[u,v])=(\kk^\times)^2 \rtimes \{\omega\}$
where $\omega$ is the automorphism that exchanges the generators $u$ and $v$ \cite[Proposition 1.4.4]{AC}.
Similarly, $\Aut(\wa)=\kk^\times$ for $\mu \neq \pm 1$ and
$\Aut(A_1^{-1}) = \kk^\times \rtimes \{\omega\}$ \cite[Proposition 1.5]{AD}.
We use these facts throughout without further comment.

\begin{proposition}
\label{prop.daction}
Let $A= \kk_{\mu}[u,v]$ or $A=\wa$ with $|\mu|=m>1$. 
Then $H_n(\lambda)$ acts linearly and inner faithfully on $A$ 
if and only if $m \mid n$.
Moreover, the action is given by one of the following:
\begin{enumerate}
\item $g(u) = \mu u$, $g(v) = \lambda\mu v$, $x(u) = 0$, $x(v) = \eta u$ for some $\eta \in \kk^\times$, and if $A= \wa$ then $\lambda = \mu^{-2}$; or \label{list.daction}
\item $g(u) = \lambda\mu\inv u$, $g(v) = \mu\inv v$, $x(u) = \eta v$ for some $\eta \in \kk^\times$, $x(v) = 0$, and if $A = \wa$ then $\lambda = \mu^2$.
\end{enumerate}
\end{proposition}

\begin{proof}
Note that $A = \kk\langle u,v \mid uv-\mu vu - \kappa\rangle$, 
where $\kappa \in \{0,1\}$. Because $g$ is grouplike, it acts as an automorphism on $A$.
We first assume that $g$ acts diagonally with respect to the given generators. 
That is, $g(u) = \alpha u$ and $g(v) = \beta v$ for some $\alpha, \beta \in \kk$ which are $n$th roots of $1$.
When $A = \wa$, we have the restriction that $\beta = \alpha \inv$.

Since $x$ acts linearly, its action on $\kk u + \kk v$ can be represented as a matrix, so, abusing notation, we write
\[ x = \begin{pmatrix}a_1 & b_1 \\ a_2 & b_2\end{pmatrix} \in M_2(\kk).\]
Since the action of $H_n(\lambda)$ is inner faithful, the matrix $x$ is nonzero \cite[Lemma 2.5]{KW}.  Additionally,
\begin{align}
\label{eq.xgrel}
0 = xg-\lambda gx = \begin{pmatrix} 
	a_1\alpha(1-\lambda)		& b_1(\beta-\lambda\alpha) \\
    a_2(\alpha-\lambda\beta)	& b_2\beta(1-\lambda) 
\end{pmatrix}.
\end{align}
Thus, $a_1=b_2=0$.
Now
\[ 0 = x^2 = \begin{pmatrix}b_1a_2 & 0 \\ 0 & b_1a_2 \end{pmatrix}\]
and hence, $a_2=0$ or $b_1=0$.

If $a_2 = 0$, then $x(u) = 0$ and $x(v) = b_1 u$. Furthermore,
\[ 0 = x(uv-\mu vu-\kappa) = (\alpha - \mu)b_1 u^2.\]
Thus, $\alpha = \mu$ and so by \eqref{eq.xgrel}, $\beta=\lambda\mu$. 
In the case of $\wa$, this implies $\lambda = \mu^{-2}$. Similarly, if $b_1 = 0$, then
$x(u) = a_2v$, $x(v) = 0$ and $(1-\mu\beta)a_2 v^2 = 0$ so $\beta = \mu \inv$ wherein $\alpha = \lambda\mu\inv$. In the case of $\wa$, this implies $\lambda = \mu^2$. In either case, to satisfy $g^n=1$, we must have $m \mid n$.

We now show that there are no faithful linear actions with $g$ acting non-diagonally on the given generators. 
Suppose that $\mu = -1$ and that there exists such an action.
Since $g$ is a non-diagonal automorphism, $g(u) = \alpha v$ and $g(v) = \beta u$ for some $\alpha, \beta \in \kk^\times$. 
As before, let $x(u) = a_1 u + a_2v$ and $x(v) = b_1 u + b_2v$. Then
\begin{align}
\notag		0 &= x(uv+vu-\kappa) \\
\label{eq.xrel}	&= (\alpha b_2 + a_2)v^2 + (\beta a_1 + b_1)u^2
        	+ (a_1+\beta a_2 - \alpha b_1 - b_2)uv
            + \kappa(\alpha b_1 + b_2).
\end{align}
Hence, $b_2 = - \alpha \inv a_2$ and $b_1 = - \beta a_1$. Then
%Further, if $a_2 = b_1 = 0$ then $a_1 = b_2 = 0$, contradicting $x \neq 0$. We also have
\begin{align} \label{eq.xrel2}
0 	= x^2 
	= \begin{pmatrix}a_1^2 + a_2b_1 & b_1(a_1 + b_2) \\ a_2(a_1 + b_2) & a_2b_1 + b_2^2 \end{pmatrix}
    = \begin{pmatrix}a_1(a_1-\beta a_2) & \beta a_1 ( \alpha\inv a_2-a_1) \\
    	a_2(a_1-\alpha\inv a_2) & a_2(\alpha^{-2}a_2-\beta a_1)
    \end{pmatrix}.
\end{align}
Since $x \neq 0$, then it now follows from \eqref{eq.xrel} and \eqref{eq.xrel2} 
that all parameters are nonzero.
From (2.4) we have $a_2=\alpha a_1$ and $\alpha\beta=1$.
Combining this with our deductions for $b_1$ and $b_2$ and 
substituting into the $uv$ coefficient of \eqref{eq.xrel} yields
%Since $a_2$ and $b_1$ are not both zero therefore $b_2 = -a_1$. 
$4a_1=0$, a contradiction.
%\[ 0 = a_1+\beta a_2 - \alpha b_1 - b_2 = 2a_1(1 + \alpha \beta).\]
%If $a_1 = 0$ then $x = 0$, which is a contradiction. Hence, $\alpha\beta = - 1$.
%
%Now,
%\[ xg-\lambda gx = \begin{pmatrix} 
%	\beta b_1-\lambda\alpha a_2 & \alpha(a_1-\lambda b_2) \\ 
%	\beta(b_2-\lambda a_1)		& \alpha a_2-\lambda\beta b_1
%\end{pmatrix}. \]
%Thus, $a_1=\lambda b_2 = \lambda (\lambda a_1)$ and so $\lambda = \pm 1$. 
%Since $H_1(1)$ is a trivial Hopf algebra, $\lambda = -1$. Hence $g^2 = 1$. 
%But then $\alpha \beta = 1$ is a contradiction.
\end{proof}

\begin{remark} 
\label{rem.haction}
Since the definitions of $\kk_{\mu}[u,v]$ and $\wa$ are symmetric in $u$ and $v$ 
up to a scalar, we may assume without loss of generality that the $H_n(\lambda)$
action is given by \eqref{list.daction} above.
In addition, the condition $\lambda=\mu^{-2}$ in the case of $\wa$ forces $|\mu|=n$ and $n$ to be odd.
We assume this henceforth without comment.
\end{remark}

Denote the {\sf $\lambda$-binomial coefficient} by
\[ 
\qbinom{k}{i}_\lambda
	= \frac{(1-\lambda^k)(1-\lambda^{k-1})\cdots(1-\lambda^{k-i+1})}{(1-\lambda)(1-\lambda^2)\cdots(1-\lambda^i)},
\]
for $k \geq 1$ and $0 \leq i \leq k$.
For shorthand, we let
\[ [n]_\lambda = \frac{1-\lambda^n}{1-\lambda} = 1+\lambda+ \cdots + \lambda^{n-1}\]
and 
\[ [n]_\lambda ! 
	= [1]_\lambda [2]_\lambda \cdots [n-1]_\lambda [n]_\lambda 
	= 1 (1+\lambda) (1+\lambda + \lambda^2) \dots (1 + \lambda + \dots + \lambda^{n-1}).\]
By \cite[Lemma 7.3.1]{Rad},
\begin{align}
\label{eq.coprod}
\Delta(g^\ell x^k) 
	= \sum_{i=0}^k \qbinom{k}{i}_\lambda g^{\ell + i} x^{k-i} \tensor g^{\ell} x^i.
\end{align}

\begin{lemma}
\label{lem.fring}
Let $|\mu| = m > 1$ and let $A= \kk_{\mu}[u,v]$ or $A = \wa$.
Suppose $H=H_n(\lambda)$ acts linearly and inner faithfully on $A$.
\begin{enumerate}
\item $A^H = \kk[u^m,v^n]$.
\item $A\#H$ is prime if and only if $m=n$.
\end{enumerate}
\end{lemma}

\begin{proof}
(1) By Proposition \ref{prop.daction},
$m \mid n$ and so $\lambda^\ell = \mu$ for some $\ell \in \NN$.
Thus, $g(u^i v^j) = \mu^{i+j(\ell+1)} u^iv^j$. It follows that 
\[A^{\langle g \rangle} = \Span_\kk\{ u^iv^j : i+j(\ell+1) \equiv 0 \mod m\}.\]

Now we compute $A^{\langle x \rangle}$.
First suppose that $A = \qp$.
We claim inductively that $x(u^k) = 0$. For $k \geq 1$,
\begin{align*}
x(u^k) = x(u)u^{k-1} + g(u)x(u^{k-1}) = 0.
\end{align*}

Assume inductively for $k \geq 1$ we have
$x(v^k) = \eta [k]_\lambda uv^{k-1}$. Then
\begin{align*}
x(v^{k+1})
	&= x(v)v^k + g(v)x(v^k) \\
	&= \eta uv^k + (\lambda\mu v)(\eta [k]_\lambda uv^{k-1}) \\
	&= \eta \left( uv^k + \lambda [k]_\lambda uv^k \right) \\
	&= \eta [k+1]_\lambda uv^k.
\end{align*}
Thus, $x(v^n) = 0$. We also have
\begin{align}
\label{eq.xaction}
x(u^iv^j)
	= x(u^i)v^j + g(u^i)x(v^j)
    = (\mu^i u^i)(\eta [j]_\lambda uv^{j-1})
	= \mu^i \eta [j]_\lambda u^{i+1} v^{j-1}.
\end{align}
It follows that 
$A^{\langle x \rangle} = \Span_\kk\{ u^iv^j : j \equiv 0 \mod n \} = \kk[u,v^n]$
and so
\[ A^H = A^{\langle g \rangle} \cap A^{\langle x \rangle} 
	= \Span_\kk\{ u^iv^j : i \equiv 0 \mod m, j \equiv 0 \mod n \}
    = \kk[u^m,v^n]
\]

Next suppose $A=\wa$ and recall that $\lambda=\mu^{-2}$, so $m>2$.
An induction as above shows that $x(u^k)=0$ for all $k$. In addition,
\[ x(v^k) = \eta \left( [k]_\lambda uv^{k-1} - \lambda\qbinom{k}{2}_{\mu\inv} v^{k-2}\right).\]
The induction step is similar and requires use of the following identity,
\[ \qbinom{k}{1}_{q^2} + q\qbinom{k}{2}_q = \qbinom{k+1}{2}_q. \]

(2) By \cite[Corollary 10]{B}, $A\#H$ is prime if and only if there exists
$0 \neq a \in A^{\langle x \rangle}$ such that $g(a) = \lambda^{n-1} a$.
As $g$ is a linear automorphism, it suffices to check this condition
on monomials in $A^{\langle x \rangle}$.
We have $A^{\langle x \rangle} = \kk[u,v^n]$ by (1) and so
$g(u^i v^{kn}) = \mu^i (\lambda\mu)^{kn} u^i v^{kn} = \mu^i u^i v^{kn}$.
Now $\mu^i=\lambda^{n-1}$ if and only if $\mu^i = \lambda\inv$ for some $i$,
and such an $i$ exists if and only if $n \mid m$.
As we already have assumed $m \mid n$, then the result follows.
\end{proof}

Our primary interest will be in computing the discriminant of $A \# H_n(\lambda)$ over its center when $A=\qp$ or $A=\wa$.
We now study the center of the smash product.

Given an $H_n(\lambda)$-module algebra $R$, denote by 
$R(k) = \{ r \in R \mid g(r) = \lambda^k r \}$ the 
$\lambda^k$-weight space of the $g$-action on $R$.

\begin{lemma}
\label{lem.cnt1}
Let $H=H_n(\lambda)$ and let $R$ be an $H$-module algebra.
If $z \in \cnt(R\# H)$ with
\begin{align} \label{eq.zcnt}
z = \sum_{i=0}^{n-1} \sum_{j=0}^{n-1} r_{i,j} \# g^i x^j , \qquad r_{i,j} \in R,
\end{align}
then $r_{i,j} \in R(j)$ and
\begin{align}
\label{eq.cntx}
x(r_{i,j}) = \begin{cases}
	0	& j=0 \\
	(1-\lambda^{i+j-1}) r_{i,j-1} & j > 0.
\end{cases}
\end{align}
\end{lemma}

\begin{proof}
Let $z \in \cnt(R\# H)$ as in \eqref{eq.zcnt}. Then
\begin{align*}
0 &= [ 1 \# g, z]
= \sum_{i=0}^{n-1} \sum_{j=0}^{n-1} [1\#g,r_{i,j}\#g^ix^j] \\
&= \sum_{i=0}^{n-1} \sum_{j=0}^{n-1} (g(r_{i,j}) \# g^{i+1} x^j - \lambda^j r_{i,j} \# g^{i+1} x^j).
\end{align*}
Thus $R_{i,j} \in R(j)$. Now
\begin{align*}
0 &= [ 1\#x, z ] = \sum_{i=0}^{n-1} \sum_{j=0}^{n-1} [1\#x,r_{i,j}\#g^ix^j] \\
&= \sum_{i=0}^{n-1} \sum_{j=0}^{n-1} 
	\left( \lambda^i g(r_{i,j}) \#g^ix^{j+1} + x(r_{i,j})\#g^ix^j -r_{i,j}\#g^ix^{j+1} \right).
\end{align*}
Therefore, $r_{i,0}=0$ and
\begin{align*}
\sum_{i=0}^{n-1} \sum_{j=1}^{n-1} x(r_{i,j}) \# g^i x^j &= \sum_{i=0}^{n-1} \sum_{j=1}^{n-1}
	\left( r_{i,j-1} - \lambda^i g(r_{i,j-1}) \right) \# g^i x^j \\
&= \sum_{i=0}^{n-1} \sum_{j=1}^{n-1} 
	\left(r_{i,j-1} - \lambda^{i+j-1} r_{i,j-1} \right) \# g^i x^j.
\end{align*}
Thus, \eqref{eq.cntx} follows.
\end{proof}

%\begin{remark}
%\label{rmk.wtsp}
%An immediate implication of \eqref{eq.cntx} is that $x(A(0)) = 0$
%and $x(A(j)) \subset A(j-1)$ for $j>1$.
%Suppose $i+j \not\equiv 1 \mod n$, where $n=|\lambda|$.
%If $x(a_{i,j})=0$ or $0 \neq x(a_{i,j}) \notin A(j-1)$, 
%then we can infer that $a_{i,j-1}=0$.
%\end{remark}

\begin{theorem} \label{thm.cnt}
Let $H=H_n(\lambda)$ and let $R$ be an $H$-module algebra that is a domain.
Suppose that the action of $H$ on $R$ is inner faithful and no nontrivial power of $g$ is 
$X$-inner when restricted to $R^{\langle x \rangle}$.
Then $\cnt(R\# H) = \cnt(R) \cap R^H$.
\end{theorem}

\begin{proof}
Let $z \in \cnt(R\# H)$ and write as in \eqref{eq.zcnt}.
By \eqref{eq.coprod} we have for $w \in R^{\langle x \rangle}$,
\begin{align*}
0 	&= [z,w\#1] = \sum_{i=0}^{n-1} \sum_{j=0}^{n-1} [r_{i,j} \# g^i x^j,w\#1] \\
	&= \sum_{i=0}^{n-1} \sum_{j=0}^{n-1} \left( r_{i,j} g^{i+j}(w) \# g^i x^j - wr_{i,j} \# g^i x^j \right) \\
	&= \sum_{i=0}^{n-1} \sum_{j=0}^{n-1} \left( r_{i,j} g^{i+j}(w) - w r_{i,j} \right) \# g^ix^j.
\end{align*}
Thus,
\begin{align}
\label{eq.cntu}
wr_{i,j} = r_{i,j}g^{i+j}(w).
\end{align}
By \eqref{eq.cntx}, $r_{i,0} \in R^{\langle x \rangle}$ for all $i$.
By hypothesis $g^i$ is not $X$-inner when $0 < i < n$ and so there exists 
$w \in R^{\langle x \rangle}$ not satisfying \eqref{eq.cntu} for $r_{i,0}$. 
Because $R$ is a domain, we must have $r_{i,0}=0$ when $0 < i < n$.
A similar argument shows $r_{0,1}=0$.

By the inner faithful hypothesis and because $x$ is nilpotent, 
there exists $y \in R$ such that $x(y) \neq 0$ and $x^k(y)=0$ for $k>1$. Thus
\begin{align*}
0 &= [z,y\#1] = \sum_{i=0}^{n-1} \sum_{j=0}^{n-1} [r_{i,j} \# g^i x^j,y\#1] \\
&= \sum_{i=0}^{n-1} \sum_{j=0}^{n-1} \left( 
	r_{i,j} g^{i+j}(y) \# g^ix^j + r_{i,j}[j]_\lambda g^{i+j-1}x(y) \# g^ix^{j-1} - yr_{i,j}\# g^i x^j \right)
\end{align*}
where $x\inv=0$.
Thus, $yr_{i,n-1} - r_{i,n-1}g^{i-1}(y) = 0$ and
\[
\sum_{i=0}^{n-1} \sum_{j=0}^{n-2} yr_{i,j} - r_{i,j} g^{i+j}(y) \# g^i x^j
= \sum_{i=0}^{n-1} \sum_{j=1}^{n-1} [j]_\lambda r_{i,j} g^{i+j-1}x(y) \# g^i x^{j-1}.
\]
Renumbering the right hand side, we get
\[
\sum_{i=0}^{n-1} \sum_{j=0}^{n-2} yr_{i,j} - r_{i,j} g^{i+j}(y) \# g^i x^j
= \sum_{i=0}^{n-1} \sum_{j=0}^{n-2} [j+1]_\lambda r_{i,j+1} g^{i+j}x(y) \# g^i x^j.
\]
This gives our final relation
\begin{align}
\label{eq.cntv2}
yr_{i,j} - r_{i,j} g^{i+j}(y) = [j+1]_\lambda r_{i,j+1} g^{i+j}x(y) \qquad\text{for } 0 \leq j \leq n-2.
\end{align}

Suppose $r_{0,k}=0$ for some $k\geq 1$.
Since $|\lambda|=n$, $g^{i+j}x(y) \neq 0$ by hypothesis.
As $R$ is a domain, then by \eqref{eq.cntv2}
we have $r_{0,k+1}=0$.
Above we showed $r_{0,1}=0$ and thus it follows
from induction that $r_{0,k}=0$ when $1 \leq k < n$.
Since $r_{i,0}=0$ when $0 < i < n$,
then a similar argument shows that
$r_{i,j}=0$ when $0 < i < n$ and $0 \leq j < n$.
Thus, $\cnt(R\# H) \subset R$ and the result follows.
\end{proof}

\begin{corollary}
\label{cor.qpcnt}
Let $A=\kk_{\mu}[u,v]$ or $A=\wa$. 
Suppose $H=H_n(\lambda)$ acts linearly and inner faithfully on $A$.
If $|\mu|=n>1$, then $\cnt(A \# H) = A^H = \kk[u^n,v^n]$.
\end{corollary}

\begin{proof}
By the proof of Lemma \ref{lem.fring}, $A^{\langle x \rangle}=\kk[u,v^n]$. 
As no non-identity power of $g$ acts trivially on $u$,
then we may apply Theorem \ref{thm.cnt}.
The result now follows from Lemma \ref{lem.fring} (1).
\end{proof}

\begin{remark}
If we loosen the hypotheses in Corollary \ref{cor.qpcnt}
then it may no longer hold. For example, consider $H_2(-1)$
acting on the commutative polynomial ring $A=\kk[u,v]$ with action
given by Remark \ref{rem.haction}.
Then the element $u\#g + 2v\#gx$ is central.
However, this does not violate the conditions of Theorem \ref{thm.cnt}
because $g$ is inner on $A^{\langle x \rangle} = \kk[u,v^n]$.
\end{remark}

\begin{question}
Let $A=\qp$.
Assume the hypotheses of Corollary \ref{cor.qpcnt} but assume
$|\mu|>1$ properly divides $|\lambda|$.
Is it still true that $\cnt(A\#H_n(\lambda)) \subset A$?
\end{question}

Although our interest is primarily in the algebras $\wa$ and $\qp$, the following two examples show that our techniques can be used to determine the centers 
in other cases as well.

\begin{example}
Let $\lambda$ and $\mu$ be primitive third roots of unity.
Define a quantum affine $3$-space $A$ on generators $u,v,w$ subject to the relations
\[ uv=\mu vu, \quad vw=\lambda\mu wv, \quad wu=\mu uw.\]
There is an action of $H_3(\lambda)$ on $A$ given by
\begin{align*}
&g(u)=\mu u, & &g(v) = \lambda\mu v, & &g(w) = \lambda^2\mu w, \\
&x(u) = 0, & &x(v) = u, & &x(w) = v.
\end{align*}
The details are left to the reader.
%We check the $x$ action below.
%\begin{align*}
%x(uv-\mu vu) &= \mu u^2 - \mu(u^2) = 0 \\ 
%x(vw-\lambda\mu wv) &= (\lambda\mu v^2 + uw) - \lambda\mu(\lambda^2\mu wu + v^2) = uw - \mu^2 wu = 0 \\
%x(wu-\mu uw) &= (vu) - \mu(\mu uv) = 0.
%\end{align*}
Of course, \eqref{eq.xaction} still holds and extends in an obvious way to
\[
x(u^iv^jw^k) = \lambda^j \mu^{i+j} [k]_\lambda u^iv^{j+1}w^{k-1}
	+ \mu^i [j]_\lambda u^{i+1}v^{j-1}w^k.
\]
Thus, we have $A^{\langle x \rangle} = \kk[u,v^3,w^3]$ and 
hence $A^{H_3(\lambda)}=\kk[u^3,v^3,w^3]$.
Thus, by Theorem \ref{thm.cnt}, $\cnt(A\#H)=A^{H_3(\lambda)}$.
\end{example}

\begin{example}
Let $\mu$ be a primitive $n$th root of unity, $n$ odd and $n > 2$.
The quantum coordinate ring of $2 \times 2$ matrices $\mq$ 
is generated by $a,b,c,d$ subject to the relations
\begin{align*}
	&ab={\mu}ba,	&	&bd={\mu}db,	&	&bc=cb,			\\
	&ac={\mu}ca,	&	&cd={\mu}dc,	&	&ad-da=(\mu-\mu\inv)bc.
\end{align*}
Let $H=H_n(\lambda)$ and $\lambda=\mu^{-2}$,
whence there is an action of $H$ on $\mq$ given by
\begin{align*}
	&g(a)={\mu}a,	&	&g(b)={\mu}b,	&	&g(c)={\mu\inv}c,	&	&g(d)={\mu\inv}d, \\
	&x(a)=0,	&	&x(b)=0,	&	&x(c)=a,	&			&x(d)=b.
\end{align*}
The details are left to the reader.
%It is clear that $x^2=0$ and $g^n=1$ provided $|\mu| \mid |\lambda|$.
%Moreover, $xg-\lambda gx=0$ under the given condition $\lambda=\mu^{-2}$.
%It is also clear that $g$ is an automorphism of $\mq$. It is left to check the action of $x$.
%\begin{align*}
%	x(ab-\mu ba) &= 0 \\
%	x(ac-\mu ca) &= (\mu a)(a) - \mu (a) a  = 0 \\
%	x(bd-\mu db) &= (\mu b)(b) - \mu (b) b = 0 \\
%	x(bc-cb)	&= (\mu b)(a) - (a)b = 0 \\
%	x(ad-da-(\mu-\mu\inv)bc)	&= (\mu a)(b) - (b)a - (\mu-\mu\inv)(\mu b)(a)
%		= (\mu^2-1)ba - (\mu^2-1)ba = 0 \\
%	x(cd-\mu dc)	&= ((\mu\inv c)(b) + (a)d) - \mu( (\mu\inv d)(a) + (b)c) 
%		= ad-da-(\mu-\mu\inv)bc = 0.
%\end{align*}

As in the case of $\wa$, we assume $|\mu|=n > 1$ with $n$ odd.
We have $\mq^{\langle x \rangle}=\kk[a,b,c^n,d^n]$ and
$\mq^H = \kk[a^i b^j,c^n,d^n : i+j \equiv 0 \mod n]$. 
To show this, we apply induction as in Lemma \ref{lem.fring}. 
Let $[k]=[k]_{\mu^{-2}}$. 
We have $x(c^k)=[k]ac^{k-1}$ and $x(d^\ell)=[\ell]bd^{\ell-1}$.
Thus, $x(c^k)=0$ and $x(d^\ell)=0$ if and only if $k,\ell \equiv 0 \mod n$.
Additionally, $g(c^n)=c^n$ and $g(d^n)=d^n$, whence $c^n,d^n\in \mq^H$.
As $x(a)=x(b)=0$ and $g(a^ib^j)=\mu^{i+j}a^ib^j$, 
we have that $a^ib^j\in \mq^H$ whenever $i+j\equiv 0 \mod n$. 
It follows that
\begin{align*}
x(a^ib^jc^kd^\ell)&=x(a^ib^jc^k)d^\ell+g(a^ib^jc^k)x(d^\ell) \\
	&= g(a^ib^j)x(c^k)d^\ell+g(a^ib^jc^k)d^\ell \\
	&= \mu^{i+j}a^ib^j\left([k]ac^{k-1}d+\mu^{-k}[\ell]c^kbd^{\ell-1}\right).
\end{align*}
By the PBW theorem for quantum matrices \cite{BG}, 
$x(a^ib^jc^kd^l)$ is nonzero if either $k$ or $\ell$ is not a multiple of $n$.
%Since $k,l<n$ hence both $\gamma_k$ and $\gamma_l$ are nonzero, and $\{a^ib^jc^{k-1}ad,a^ib^jd^{l-1}b\}$ is linearly independent for any $k,l\leq\tfrac{n-1}{2}$ ($i+j=k+l$ is due to the action $g$ hence $k,l\leq k+l\leq\tfrac{n-1}{2}$), Therefore $x(a^ib^jc^kd^l)$ is nonzero whenever $i+j+k+l<n$. 
Therefore, $\mq^{\langle x \rangle}$ is generated by
$\{a,b,c^n,d^n\}$ and so $\mq^H$ is generated by
$\{c^n,d^n,a^ib^j|\,i+j=n\ \text{and}\ i,j\geq0\}$.

The action of $H$ on $\mq$ is inner faithful
and as no non-trivial power of $g$ is $X$-inner on $\mq^{\langle x \rangle}$, 
then we may apply Theorem \ref{thm.cnt} to obtain that $\cnt(\mq\# H) = \mq^H$.
\end{example}

\section{Induced Poisson structures and discriminant computations}
\label{sec.disc}

Let $A$ be an algebra finitely generated and free of rank $\omega$ over 
a central subalgebra $C \subseteq \cnt(A)$.
The {\sf regular trace} is defined as the composition
$\tr_{\reg}:A \xrightarrow{\lm} M_n(C) \xrightarrow{\tr_\int} C$
where $\lm$ denotes left multiplication 
and $\tr_\int$ the usual (internal) trace.
Throughout, we use the notation $\tr$ to denote $\tr_{\reg}$.
If $Z:=\{z_i\}_{i=1}^\omega$ is a (finite) basis of $A$ over $C$,
then the {\sf discriminant of $A$ over $C$} is defined to be
\begin{align}
\label{eq.disc}
d(A/C) =_{C^\times} \det(\tr(z_iz_j))_{\omega \times \omega} \in C.
\end{align}
There is an intimate connection between discriminants and 
the Poisson structures induced on centers of specializations \cite{NTY}.

Let $R$ be a $\kk[q^{\pm 1}]$-algebra.
The {\sf specialization} $R_\epsilon$ of $R$ at $\epsilon \in \kk^\times$ 
is defined as $\Rep := R/(q-\epsilon)$ \cite{BG2}.
The canonical projection $\sigma:R \rightarrow \Rep$ induces a Poisson
structure on $\cnt(R_\epsilon)$ via
\[ 
\{ \sigma(x_i),\sigma(x_j) \} 
	= \sigma\left( \frac{x_ix_j-x_jx_i}{q-\epsilon} \right),
    \quad x_i, x_j \in \sigma\inv(\cnt(R_\epsilon)).
\]

Throughout, we assume all Poisson algebras are commutative.
We will be interested in the induced Poisson structure on the center
of a certain Ore extension related to the smash product $A\#H_n(\lambda)$
when $A=\qp$ or $A=\wa$.
Our primary goal in this section is to compute the discriminant of that Ore 
extension using the Poisson techniques developed in \cite{NTY}.
We then extend this to determine the discriminant of the smash product itself.

We begin by giving background on Poisson Ore extensions \cite{OH1,OH2}.
Let $B$ be a Poisson algebra with Poisson bracket $\{\,,\}_B$.
A derivation $\alpha$ of $B$ is a {\sf Poisson derivation} provided 
\begin{align}
\label{eq.aderiv}
	\alpha(\{a,b\}_B) = \{\alpha(a),b\}_B + \{a,\alpha(b)\}_B \quad
    \text{for all } a,b \in B.
\end{align}
A derivation $\beta$ of $B$ is an $\alpha$-derivation 
(where $\alpha$ is a Poisson derivation) provided
\begin{align}
\label{eq.bderiv}
	\beta(\{a,b\}_B) - \{\beta(a),b\}_B - \{a,\beta(b)\}_B
	= \beta(a)\alpha(b) - \alpha(a)\beta(b) \quad
    \text{for all } a,b \in B.
\end{align}
Given a Poisson derivation $\alpha$ and an $\alpha$-derivation $\beta$,
the {\sf Poisson Ore extension} $B[z;\alpha,\beta]_P$
is the polynomial ring $B[z]$ with Poisson bracket
\[ 
\{a,b\} = \{a,b\}_B, \qquad
\{z,a\} = \alpha(a)z+\beta(a) \quad \text{for all } a,b \in B.
\]
The Poisson derivation $\alpha$ is {\sf inner}
if there exists a unit $a \in B$ such that $\alpha(b) = a\inv\{b,a\}$ for all $b \in B$.
Furthermore, the $\alpha$-derivation $\beta$ is {\sf $\alpha$-inner} if there exists $d \in B$ 
such that $\beta(b)=d\alpha(b)+\{b,d\}$ for all $b \in B$.
An element $y$ in a Poisson algebra $B$ is said to be {\sf Poisson normal} 
if $\{y,b\} \in yB$ for all $b \in B$.
If $y \in B$ is Poisson normal and $(y)$ is a prime ideal in $B$,
then $y$ is said to be {\sf Poisson prime}.

Let $A$ be an algebra and $q \in A$ a nonzero divisor.
The pair $(\tau,\delta)$ where $\tau \in \Aut(A)$ and
$\delta$ is a $\tau$-derivation of $A$ is said to be a 
{\sf $q$-skew extension} of $A$ provided 
$\tau(q)=q$, $\delta(q)=0$, and $\tau\delta=q\delta\tau$ \cite[Section 6]{G}.
The following lemma is implied directly by \cite[Theorem 1.1]{OH1}.

\begin{lemma}
\label{lem.indp2}
Let $B[z]$ be a Poisson algebra such that $\{z,a\} \in Bz + B$ for all $a \in B$. 
Then $z$ induces a Poisson derivation $\alpha$
and an $\alpha$-derivation $\beta$ satisfying \eqref{eq.bderiv} such that
$\{z,b\} = \alpha(b)z+\beta(b)$ for all $b \in B$.
\end{lemma}

\begin{proposition}
\label{prop.indp2}
Let $A$ be an $\kq$-algebra and
$(\tau,\delta)$ a $q$-skew extension of $A$.
\begin{enumerate}
\item The Ore extension $A[t;\tau,\delta]$ is a $\kq$-algebra and for $\epsilon \in \kk^\times$, 
\[(A[t;\tau,\delta])_\epsilon = A_\epsilon[t;\tau,\delta].\]

\item Suppose $\Cep=\Bep[t^m]$ where $\Bep$ is a central subalgebra of
$\Aep$ and $m$ is the order of $\left.\tau\right|_{\Bep}$.
Then the induced Poisson structure on $\Cep$
is a Poisson Ore extension of the induced structure on $\Bep$.
In particular, for $b,b' \in \Bep$ we have
$\{b,b'\}_{\Cep}=\{b,b'\}_{\Bep}$ and $\{z,b\} = \alpha(b)z+\beta(b)$ where 
\[
\alpha(b)=\sigma\left( \frac{\tau^m(a)-a}{q-\epsilon} \right)
\quad\text{and}\quad
\beta(b)=\sigma\left( \frac{\delta^m(a)}{q-\epsilon}\right)
\quad\text{for } a \in \sigma\inv(b).\]
\end{enumerate}
\end{proposition}

\begin{proof}
(1) This is clear.

(2) Set $z=\sigma(t^m)$. Choose $b \in \Bep$ and let $a \in \sigma\inv(b)$.
\begin{align*}
\{z,b\} &= \{\sigma(t^m),\sigma(a)\} \\
    	&= \sigma\left( \frac{t^m a - at^m}{q-\epsilon} \right) \\
	& = \sigma\left( \frac{-at^m + \sum_{i=0}^m
        \left[\begin{smallmatrix}m \\ i\end{smallmatrix}\right]_q
        \tau^{m-i}\delta^i(a)t^{m-i} }{q-\epsilon}	\right).
\end{align*}
By our assumption on $\Cep$, all coefficients in the summation are zero
except when $i=0$ or $i=m$. Thus,
\[
\{z,b\}
	= \sigma\left( \frac{(\tau^m(a)-a)t^m + \delta^m(a)}{q-\epsilon} \right)
    = \sigma\left(\frac{\tau^m(a)-a}{q-\epsilon}\right)z
    	+ \sigma\left( \frac{\delta^m(a)}{q-\epsilon}\right).
\]
The result now follows from Lemma \ref{lem.indp2}.
\end{proof}

\begin{remark}
\label{rmk.kint}
In this section and beyond we assume the action of $H_n(\lambda)$ on $\qp$
and $\wa$ is given as in Proposition \ref{prop.daction}. 
By scaling variables, we may assume $\eta=1$.
We assume that $|\mu|=n > 1$.
Recall that in the case of $\wa$ we have $\lambda=\mu^{-2}$ and so $n$ is odd.
In general we write $\lambda = \mu^k$ for some $k$ relatively prime to $n$.
Note that we will never have $k=0$.
\end{remark}

Let $A$ be the $\kq$-algebra on generators $u,v$ subject to the relation
$uv-qvu-\kappa$ for $\kappa \in \{0,1\}$.
Specializing along $q=\mu$, we have $A_\mu \iso \qp$ (resp. $\wa$) when $\kappa=0$ (resp. $1$).
Set $R=A[x;\tau,\delta]$ with 
$\tau(u)=q u$, $\tau(v)=q^{k+1} v$, $\delta(u)=0$, $\delta(v)=u$.
By Proposition \ref{prop.indp2}, $R$ is again a $\kq$-algebra and
$R_\mu = A_\mu[x;\tau,\delta]$.
Thus, $R$ is the $\kk[q^{\pm 1}]$-algebra on generators $u,v,x$ subject to the relations
\[ uv- q vu - \kappa, \quad xu-qux, \quad xv- q^{k+1} vx-u, 
\qquad \kappa \in \{0,1\}.\]
Note that for $\epsilon \in \kk^\times$, the specialization $\Rep=R/(q-\epsilon)$ 
is an analog of the Heisenberg enveloping algebra \cite{gadhberg,KS}
and $R_\mu \iso A_\mu[x;\tau,\delta]$.
We will compute the discriminant of $R_\mu$ 
over the central subalgebra $C_\mu=\kk[u^n,v^n,x^n]$.

\begin{proposition}
\label{prop.coeffs}
Keep the above notation and hypotheses.
Set $z_1=u^n$, $z_2=v^n$ and $z_3 = x^n$.
Let $\sigma:R \rightarrow R_\mu$ be the canonical projection and define
\begin{align*}
	b_1 &= \left.\frac{q^{n^2}-1}{q-\mu}\right|_{q=\mu}, &
    	&b_2= \left.\frac{[n]_q!}{q-\mu}\right|_{q=\mu}, \\
	c_1 &= (k+1)b_1,  &
    	&c_2= \left.\frac{(-1)^{n+1} [n]_{q^k}!}{q-\mu}\right|_{q=\mu}.        
\end{align*}
By Proposition \ref{prop.indp2},
the induced Poisson structure on $C_\mu$ is given by
\begin{align*}
\qp:&	&	
	&\{z_1,z_2\} = b_1 z_1z_2, &
	&\{z_3,z_1\} = b_1 z_1z_3, &
    &\{z_3,z_2\} = c_1 z_2z_3 + c_2 z_1, \\
\wa:&	&	
	&\{z_1,z_2\} = b_1 z_1z_2 + b_2, &
	&\{z_3,z_1\} = b_1 z_1z_3, &
    &\{z_3,z_2\} = c_1 z_2z_3 + c_2 z_1.
\end{align*}
Moreover, $C_\mu= B_\mu[z_3;\alpha,\beta]_P$ where $\alpha,\beta$ are given by
\[
\alpha(z_1)=b_1 z_1,  \quad \alpha(z_2)=c_1 z_2, \qquad
\beta(z_1)=0, \quad \beta(z_2)=c_2 z_1.
\]
\end{proposition}

\begin{proof}
The bracket for $z_1$ and $z_2$ was computed in \cite[Theorem 3.4]{NTY}. Observe that
\begin{align*}
c_1 &= \left.\frac{q^{n^2(k+1)}-1}{q-\mu}\right|_{q=\mu} \\
 &= \left.\frac{(q^{n^2}-1)(1 + q^{n^2} + \cdots + q^{kn^2})}{q-\mu}\right|_{q=\mu} \\
 &= (k+1)\left.\frac{q^{n^2}-1}{q-\mu}\right|_{q=\mu}.
\end{align*}

It is left only to compute $\beta(z_2)$.
As $v^n \in \sigma\inv(z_2)$ we have $\beta(z_2)=\sigma\left( \frac{\delta^n(v^n)}{q-\mu}\right)$
by Proposition \ref{prop.indp2}.
Since $q^k\tau\delta=\delta\tau$ then by the $q$-Leibniz rule \cite[Lemma 6.2]{G} 
and the fact that $\delta^i(v^\ell)=0$ when $i>l$, we have
\[ \delta^n(v^n)
		= \sum_{i=0}^n \qbinom{n}{i}_{q^k} \tau^{n-i}\delta^i(v)\delta^{n-i}(v^{n-1})
        = [n]_{q^k} (q^{n-1}u)\delta^{n-1}(v^{n-1}).
\]
An induction now shows that $\delta^n(v^n) = \prod_{i=0}^{n-1} q^{n-i-1} [n-i]_{q^k} u^n$.
But $\displaystyle \prod_{i=0}^{n-1} \mu^{n-i-1} = (-1)^{n+1}$ and so the result follows.
\end{proof}

By Proposition \ref{prop.coeffs}, $b_1=c_1$ if and only if $k=0$ if and only if $\lambda=\mu^0=1$. 
Thus, we may safely disregard this case.
Moreover, when $\lambda=\mu^{-2}$, as in the case of $A=\wa$, we have $b_1=-c_1$.

We now employ techniques from \cite{OH1,OH2} to compute
the Poisson primes of $C_\mu$ and use this to determine
the discriminant of $R_\mu$ over $C_\mu$ in both the 
quantum plane case and the quantum Weyl algebra case.

\begin{lemma}
\label{lem.pprimes1}
Let $A = \qp$. Up to a scalar, the Poisson prime elements of $C_\mu$ are $z_1$ and $z_2z_3+ \tornado z_1$ where $\tornado = \frac{c_2}{c_1-b_1}$.
\end{lemma}

\begin{proof}
Let $\fp$ be a Poisson prime element of $C_\mu$.
Then $\{\fp,a\} = \fp\gamma(a)$ for some derivation $\gamma$ of $\kk[z_1,z_2,z_3]$.
Suppose that $\fp \in \kk[z_1,z_2]$. 
We claim that $\fp = \kappa z_1$ for some $\kappa \in \kk$.

Write $\fp = \sum_{i=0}^\ell f_i z_2^i$ with $f_i \in \kk[z_1]$. Then
\begin{align*}
	\fp\gamma(z_1) = \{ \fp, z_1 \} = \sum_{i=0}^\ell \{ f_i z_2^i, z_1 \}
    	= \sum_{i=0}^\ell f_i \{ z_2^i , z_1 \}
        = \sum_{i=0}^\ell f_i (-ib_1 z_1z_2^i).
\end{align*}
For each $1 \leq i \leq \ell$ we have $f_i\gamma(z_1) = -b_1i f_i$, 
so at most one of the $f_i$ is nonzero. 
Therefore, $\fp = f z_2^\ell$ for some $f \in \kk[z_1]$, $\ell \in \NN$. 
By a symmetric argument, we can conclude that in fact $\fp = \kappa z_1^k z_2^\ell$ for some $\kappa \in \kk$, $k, \ell \in \NN$.
Since $\fp$ is a Poisson prime, it is also a prime element of $\kk[z_1,z_2]$, 
so $\fp = \kappa z_1$ or $\fp = \kappa z_2$. 
It is easy to check that $z_1$ is a Poisson normal element of $C_\mu$ and that $z_2$ is not,
which completes our proof in this case.

It is now left to determine the remaining Poisson primes of $C_\mu$.
Let $Q=\kk(z_1,z_2)$, the quotient field of $\kk[z_1,z_2]$. 
The corresponding localization extends $\alpha$ and $\beta$ uniquely to a Poisson Ore extension $Q[z_3; \alpha, \beta]_P$.
Let $\tornado = \frac{c_2}{c_1-b_1}$ and set $d = \tornado z_1 z_2\inv \in Q$.
We claim that for all $r \in Q$ we have $\beta(r) = d\alpha(r) + \{d,r\}$.
It suffices to check this on the generators $z_1,z_2$,
\begin{align*}
d\alpha(z_1) + \{z_1,d\} 
	&= \tornado \left(z_1 z_2\inv (b_1z_1) + \{z_1,z_1z_2\inv\}\right)
    = \tornado \left( b_1z_1^2 z_2\inv + z_1\{z_1,z_2\inv\}\right) \\
    &= \tornado \left( b_1z_1^2 z_2\inv + z_1(-b_1z_1z_2\inv)\right)
	= 0 = \beta(z_1), \\
d\alpha(z_2) + \{z_2,d\} 
	&= \tornado ( z_1 z_2\inv (c_1 z_2) + \{z_2, z_1z_2\inv \} )
	= \tornado ( c_1 z_1 +  z_2\inv \{z_2, z_1\} ) \\
    &= \tornado ( c_1 z_1 +  z_2\inv (-b_1z_1z_2 ) )
    = \tornado ( c_1 - b_1) z_1
	= c_2 z_1 = \beta(z_2).
\end{align*}
Thus, $Q[z_3;\alpha,\beta]_P = Q[z_3+d;\alpha]_P$.

It follows that $z_3+d$ is a Poisson normal element and clearing fractions 
gives that $z_2z_3+ \tornado z_1$ is a Poisson prime element in $P$.
On the other hand, $Q$ is $\alpha$-simple (because it is simple)
and $n\alpha$ is not a inner derivation for any $n$.
Thus, $Q[(z_3+d)^{\pm 1};\alpha]$ is Poisson simple \cite[Lemma 3.3]{OH1} 
and therefore there are no further Poisson prime elements.
\end{proof}

The choice of integer $k$ in Remark \ref{rmk.kint} 
will give rise to distinct Poisson structures on $C_\mu$ and,
as is clear from Proposition \ref{prop.coeffs}, different coefficients.
Let $\tornado$ be as in Lemma \ref{lem.pprimes1}.
Set $b_1'$, $c_1'$, $c_2'$, and $\tornado'$ to be the coefficients 
obtained by replacing $k$ with $k+n$.
Then $b_1=b_1'$ and $c_1'=(k+n+1)b_1'=(k+n+1)b_1$.
A computation shows that $c_2'=\frac{k+n}{k}c_2$ so that
\[ 
\tornado' = \frac{c_2'}{c_1'-b_1'} = \frac{\frac{k+n}{k}c_2}{(k+n)b_1} = \frac{c_2}{kb_1} = \frac{c_2}{c_1-b_1} = \tornado.
\]
Thus, the Poisson prime element $z_2z_3 + \tornado z_1$ 
in Lemma \ref{lem.pprimes1} does not depend on the choice of lift of
$k$ from $\ZZ/n\ZZ$ to $\ZZ$.

\begin{lemma}
\label{lem.pprimes2}
Let $A = \wa$. Then, up to a scalar, the only Poisson prime element of $C_\mu$ 
is $z_1z_2z_3+ \tornado z_1^2 + b_2b_1\inv z_3$.
\end{lemma}

\begin{proof}
It is easy to see that $z_1$ is not a Poisson normal element in this case and 
so we may pass to $\kk[z_1^{\pm 1},z_2]$. 
Set $d = b_2b_1\inv z_1\inv$ so $b_2 = d(b_1 z_1)$. 
We make the change of variable $y = z_2 + d$ and
the Poisson bracket on $\kk[z_1^{\pm 1},y]$ is then $\{z_1,y\} = b_1z_1y$.

We now pass to $Q=\kk(z_1,z_2)$. Since $c_1=-b_1$ in this case then
\[
\{z_3,y\} 
	= \{z_3,z_2\} + b_2b_1\inv\{z_3,z_1\inv\}
	= (c_1z_2z_3+c_2z_1) - b_2z_1\inv z_3 
    = c_1yz_3 + c_2z_1.
\]
Thus, $y$ is not Poisson normal.
Our computations from Lemma \ref{lem.pprimes1} now show that the 
only Poisson prime element is $z_3 + \tornado z_1 y\inv$.
Clearing fractions, it follows that the only Poisson prime element of 
$C_\mu$ is $z_1z_2z_3 + \tornado z_1^2 + b_2b_1\inv z_3 $ as claimed.
\end{proof}

\begin{lemma}
\label{lem.dRlambda}
Let $\alpha = n^2(n-1)$. Then
\begin{align*}
d(R_\mu/C_\mu) =_{k^\times} 
	\begin{cases}
    	z_1^{\alpha} (z_2z_3+ \tornado z_1)^{\alpha} & \text{ if } A=\qp \\
        (z_1z_2z_3+ \tornado z_1^2 + b_2b_1\inv z_3)^\alpha	& \text{ if } A=\wa.
    \end{cases}
\end{align*}
\end{lemma}

\begin{proof}
First suppose we are in the case $A=\qp$.
Then $\cB = \{ u^iv^jx^k \mid  0 \leq i, j,k \leq n-1 \}$ is a basis of $R_\mu$ over $C_\mu$.
By Lemma \ref{lem.pprimes1}, the Poisson primes of $C_\mu$ are $z_1$ and $z_2z_3+ \tornado z_1$.
Thus, using the definition of the discriminant as at \eqref{eq.disc},
$d(R_\mu/C_\mu) =_{k^\times} z_1^{\alpha_1} (z_2z_3+ \tornado z_1)^{\alpha_2}$
for some $\alpha_1,\alpha_2 \in \NN$ \cite[Theorem 3.2]{NTY}. 
We define a grading on $R_\mu$ by setting $\deg u = 2$ and $\deg v = \deg x = 1$. Then 
\begin{align*}
2n(\alpha_1+ \alpha_2) 
	&= \deg(p_1)\alpha_1 + \deg(p_2)\alpha_1 \\
    &= 2 \sum_{b \in \cB} \deg b
	= 2 \sum_{i=0}^{n-1} \sum_{j=0}^{n-1} \sum_{k=0}^{n-1} \deg(u^iv^jx^k) \\
    &= 2\sum_{i=0}^{n-1} \sum_{j=0}^{n-1} \sum_{k=0}^{n-1} \left( 2i + j + k\right)
    = 2(2n^3(n-1)).
\end{align*}
%\[ 2n(\alpha_1+ \alpha_2) = \deg(p_1)\alpha_1 + \deg(p_2)\alpha_1 = 2 \sum_{b \in \cB} \deg b.\]
%Then
%\begin{align*} \sum_{b \in \cB} \deg b &= \sum_{i=0}^{n-1} \sum_{j=0}^{n-1} \sum_{k=0}^{n-1} \deg(u^iv^jx^k) = \sum_{i=0}^{n-1} \sum_{j=0}^{n-1} \sum_{k=0}^{n-1} \left( 2i + j + k\right) \\
%&= \sum_{i=0}^{n-1}\sum_{j=0}^{n-1}\left( 2ni + nj + \frac{n(n-1)}{2} \right) = \sum_{i=0}^{n-1} \left( 2n^2 i + \frac{n^2(n-1)}{2} + \frac{n^2(n-1)}{2} \right)\\
%&= \frac{2n^3(n-1)}{2} + \frac{n^3(n-1)}{2} + \frac{n^3(n-1)}{2} = 2n^3(n-1).
%\end{align*}
%Therefore, $2n(\alpha_1+\alpha_2) = 4n^3(n-1)$.
Therefore, $(\alpha_1+\alpha_2) = 2n^2(n-1)$.

Now suppose we are in the case $A=\wa$.
By Lemma \ref{lem.pprimes2} and \cite[Theorem 3.2]{NTY},
$d(R_\mu/C_\mu) =_{\kk^\times} (z_1z_2z_3+ \tornado z_1^2 + z_3)^\alpha$
for some $\alpha \in \NN$.
By \cite[Proposition 4.10]{CPWZ2}, discriminants respect filtrations.
Define a filtration corresponding to the grading on $R_\mu$ above, then 
\begin{align*}
z_1^{\alpha_1} (z_2z_3+ \tornado z_1)^{\alpha_2}
	&= \gr((z_1z_2z_3+ \tornado z_1^2 + z_3)^\alpha) \\
    	&= (z_1(z_2z_3+ \tornado z_1))^\alpha \\
	&= z_1^{\alpha} (z_2z_3+ \tornado z_1)^{\alpha}.
\end{align*}
Thus, $\alpha_1=\alpha=\alpha_2$.
\end{proof}

\begin{theorem}
\label{thm.disc}
Let $H=H_n(\lambda)$. Let $|\mu|=n>1$ and let $A = \qp$ or $A=\wa$.
If $H$ acts linearly and inner faithfully on $A$, then
$d(A\#H /A^H) =_{\kk^\times} u^{2n^4(n-1)}$.
\end{theorem}

\begin{proof}
Let $A=\wa$. The case of $\qp$ is similar.
By Lemma \ref{lem.dRlambda}, 
$d(R_\mu/C_\mu) =_{\kk^\times} (z_1z_2z_3+ \tornado z_1^2 + z_3)^\alpha$.
Set $S=R_\mu[\hat{g};\tau]$. Note that $|\tau|=n$ and no non-trivial power of $\tau$ is $X$-inner.
Let $\ell$ be the rank of $R_\mu$ as a $C_\mu$-module. By \cite[Theorem 6.1]{GKM},
\[ d(S/\cnt(S)) = (z_1z_2z_3+ \tornado z_1^2 + z_3)^{\alpha n}  \cdot \hat{g}^{\ell n}.\]
Note that
\[ A\#H \iso \frac{A[x;\tau,\delta][g;\tau]}{(x^n,g^n-1)}.\]
The result now follows from \cite[Proposition 4.7]{CPWZ2} and Lemma \ref{lem.dRlambda}.
\end{proof}

For a $\kk$-algebra $R$ which is module finite over its center $\cnt{(R)}$, 
the \textsf{Azumaya locus} $\mathcal{A}(R)$ of $R$ consists of those maximal ideals 
$\mathfrak{m}$ of $\cnt(R)$ such that $\mathfrak{m}R$ is the annihilator of an 
irreducible $R$-module of maximal dimension. 
By \cite[Theorem III.1.7]{BG}, $\mathcal{A}(R)$ is an open dense subset of $\Maxspec \cnt(R)$. 
Recently, Brown and Yakimov have related $\mathcal{A}(R)$ to the discriminant $d(R/\cnt(R))$. 
Having computed the discriminant of $A\#H$ over its center, 
we are able to characterize the Azumaya locus of $A\#H$ as a corollary.

\begin{corollary} \label{cor.azumaya}
Let $A$ and $H$ be as in Theorem~\ref{thm.disc}.
Then $\mathcal{A}(A\#H)$ is the complement of the zero locus of $(u^n)$ in $\Maxspec (\kk[u^n,v^n])$.
\end{corollary}

\begin{proof}
This is a direct consequence of \cite[Main Theorem]{BY} and Theorem \ref{thm.disc}.
\end{proof}

Henceforth, set $A = \kk_{-1}[u,v]$ and $H=H_2(-1)$ with the usual action of $H$ on $A$.
Let $S=A\#H$ and define the {\sf $H$-restricted automorphism group} of $S$ to be
those algebra automorphisms of $S$ that  fix $H$ up to scalar.
That is, those $\phi \in \Aut(S)$ such that
\begin{align}
\label{eq.raut}
\phi(1\#g)=\varepsilon \#g \quad\mbox{and}\quad  \phi(1\#x)=\tornado \#x \quad \quad\mbox{for some}\quad \varepsilon = \pm 1, \tornado \in \kk^\times.
\end{align}
%\[ \rAut(S) = \{ \phi \in \Aut(S) \mid \phi(1\#g)=\varepsilon \#g, \phi(1\#x)=\tornado \#x \mbox{ for some } \varepsilon = \pm 1, \tornado \in \kk^\times\}. \]
It is clear that $\rAut(S)$ is a subgroup of $\Aut(S)$.
As a final application of our discriminant calculation, we determine $\rAut(S)$\footnote{The computations are omitted.
The interested reader is referred to the appendix of the preprint version of
this paper, available at https://arxiv.org/abs/1707.02822.}.

We define two families of maps below, called even and odd type, respectively.
In what follows, $\alpha, \tornado \in \kk^\times$, $\varepsilon = \pm 1$, $I \subset \NN$ is a finite set of odd numbers, and for each $i\in I$, $\beta_i \in \kk$.
A map $\phi_e$ is said to be of {\sf even type} if it satisfies \eqref{eq.raut} and
\[ \phi_e(u\#1) = \alpha (u \# 1) \quad\mbox{and}\quad \phi_e(v\#1) = \alpha \left(\tornado \inv v \# 1 + \sum_{i \in I} \beta_i u^i \# x\right).\]
A map $\phi_o$ is said to be of {\sf odd type} if it satisfies \eqref{eq.raut} and
\[ \phi_o(u\#1) = \alpha \left( u\#g - 2 v \# gx\right)  \quad\mbox{and}\quad \phi_o(v\#1) = \alpha \left(\tornado \inv v\#g + \sum_{i \in I} \beta_i u^i \# gx\right).\]

%\begin{align*}
%&\underline{\text{Even type}} ~~				
%	&	&\underline{\text{Odd type}}	 \\
%&\phi_e(u\#1) = \alpha (u \# 1) 				&	&\phi_o(u\#1) = \alpha \left( u\#g - 2 v \# gx\right) \\
%&\phi_e(v\#1) = \alpha \left(\tornado \inv v \# 1 + \sum_{i \in I} \beta_i u^i \# x\right) 	&	
%&\phi_o(v\#1) = \alpha \left(\tornado \inv v\#g + \sum_{i \in I} \beta_i u^i \# gx\right) \\
%&\phi_e(1\#g) = \varepsilon \#g & & \phi_o(1\#g) = \varepsilon \#g \\
%&\phi_e(1\#x) = \tornado \#x & & \phi_o(1\#x) = \tornado\#x
%\end{align*}

The maps $\phi_e$ and $\phi_o$, when extended linearly,
define automorphisms of $S$.
Moreover, the composition of two even or two odd type automorphisms is an even automorphism,
while the composition of an even with an odd is odd.

\begin{theorem}\label{thm.raut}
Let $\phi \in \rAut(S)$. Then $\phi$ is of even type or of odd type.
\end{theorem}

\section*{Acknowledgements}
The authors wish to thank Frank Moore for helpful conversations and
computational assistance with Macaulay2,
and the referee for a close reading with many suggestions
that have improved the presentation of this article.

\appendix

\section{The restricted automorphism group of $A \# H$}
\label{sec.aut}

In this appendix we provide the necessary computations to prove Theorem \ref{thm.raut}.
We keep the notation defined above.

\begin{lemma}
When extended linearly, $\phi_e$ and $\phi_o$ define automorphisms of $S$.
Moreover, the composition of two even or two odd type automorphisms is an even automorphism,
while the composition of an even with an odd is odd.
\end{lemma}
\begin{proof}Let $\phi_e$ and $\phi_o$ be maps of even and odd type, respectively. It is routine to check that $\phi_e$ and $\phi_o$ map the relations $(u\# 1)(v\# 1) + (v\# 1)(u\# 1)$, $(1\#g)(u\# 1) + (u\# 1)(1\#g)$, $(1\#g)(v\# 1) - (v\# 1)(1\#g)$, and $(v\# 1)(1\#x) - (1\#x)(v\# 1) + u\# 1$ to zero, so $\phi_e$ and $\phi_o$ give well-defined endomorphisms of $S$. It remains to show that $\phi_e$ and $\phi_o$ are bijective.
But $\phi_e$ has an inverse of even type given by
\begin{align*} &\phi_e^{-1} (u\#1) = \alpha^{-1} (u \# 1) & 
&\phi_e^{-1} (v \# 1) = \alpha^{-1}  \left(\tornado v \# 1 - \alpha \sum_{i \in I} \beta_i \alpha^{-i} u^i \# x \right) \\
&\phi_e^{-1} (1 \# g) = \varepsilon \# g & 
&\phi_e^{-1} (1 \# x) = \tornado \inv \# x
\end{align*}
so $\phi_e$ is bijective.
%To see that $\phi_e$ is surjective, note that
%\[ v\#1 = \phi_e \left(\alpha^{-1} v \# 1 - \sum_{\text{$i$ odd }} \beta_i\alpha^{-i} u^i \# x \right).
%\]
%To show injectivity, let $f \in \ker\phi_e$. Write
%\[ f = f_1(u,v) \# 1 + f_2(u,v) \# g + f_3(u,v) \# x + f_4(u,v) \# gx.\]
%The coefficient of the identity component of $0=\phi_e(f)$ is then $f_1(\alpha u,\alpha v)$. Thus, $f_1(u,v)=0$.
%Similarly, $f_i(u,v)=0$ for $i=2,3,4$, so $f=0$.
It is easy to see that the composition of two even or two odd maps gives an even map and the composition of an even with an odd gives an odd map. Now, since $\phi_o \circ \phi_o$ is a map of even type, it is bijective. Hence, $\phi_o$ is bijective. 
\end{proof}

Our goal is to prove that every automorphism in $\rAut(S)$ is either of even or odd type.
To do this we apply our discriminant computations above.

\begin{lemma} \label{lem.autcnt} Let $\phi \in \Aut(S)$, then $\phi(u^2\#1) =_{\kk^\times} u^2 \# 1$ and
$\phi(v^2\#1) = (\kappa v^2 + f(u^2))\#1$ for some $\kappa \in \kk^\times$ and some $f \in \kk[y]$.
\end{lemma}

\begin{proof}
If $\phi \in \Aut(S)$, then $\phi$ preserves $\cnt(S)$ and hence
by \cite[Lemma 1.8]{CPWZ1}, $\phi$ preserves the ideal generated by $d(S/\cnt(S))$.
By Theorem \ref{thm.disc}, $d(S/\cnt(S))=_{\kk^\times} u^{32}$.
But $\cnt(S) = \kk[u^2,v^2]$ is a domain and so $\phi(u^2\#1) =_{\kk^\times} u^2 \# 1$.
The second claim is clear because $\phi$ must restrict to an automorphism of $\cnt(S)$.
\end{proof}

Unfortunately, the discriminant gives us no information on the Taft algebra side and
so we do not expect to be able to compute the full automorphism group from this information alone.

%Let $s \in S$ and $\phi \in \rAut(S)$. 
%Write $\phi(s) = a \# 1 + b \# g + c \# x + d \# gx$. Then
%\begin{align*}
%\phi(s)^2
%	&= (a \# 1 + b \# g + c \# x + d \# gx)^2 \\
%	&= (a^2 + bg(b) + c x(a) + d gx(b)) \# 1 \\
%		&\quad + (ab + b g(a) + c x(b) + d gx(a)) \# g \\
%		&\quad + (ac + bg(d) + c g(a) + c x(c) - db + d gx(d)) \# x \\
%		&\quad + (ad + b g(c) - c g(b) + c x(d) + da + d gx(c)  ) \# gx
%\end{align*}

\begin{hypothesis}
\label{hyp.auto}
For the remainder of the section, we write
\begin{align*}
&\phi(u\#1) = a \# 1 + b \# g + c \# x + d \# gx & &\phi(v\#1) = a' \# 1 + b' \# g + c' \# x + d' \# gx \\
&\phi(1\#g) = \varepsilon \# g & &\phi(1 \#x ) = \tornado \# x
\end{align*}
where $a, b, c, d, a' , b', c', d' \in \kk_{-1}[u,v]$, $\varepsilon = \pm 1$, and $\tornado \in \kk^\times$.
\end{hypothesis}

Because $g$ and $x$ act linearly on $A$, then $S$ is $\NN$-graded and so
we write $\phi(s)_m$ to indicate the degree $m$ component of the image of $s$ under $\phi$.
We will use similar notation for the degree components of coefficients as well.
Our first observation regards the weight spaces for the coefficients of $\phi(u\#1)$
and $\phi(v\#1)$. There are only two in this case and so we denote them by $A^+ = A(1)$
and $A^- = A(-1)$.

\begin{lemma}
\label{lem.wtsp}
Let $\phi \in \rAut(S)$.
Then $a',b',c,d \in A^+$ and $a,b,c',d' \in A^-$.
\end{lemma}

\begin{proof}
We have
\begin{align*}
0 	&= \phi((1\#g)(u\#1) + (u\#1)(1\#g)) \\
	&= (\varepsilon\#g)(a \# 1 + b \# g + c \# x + d \# gx) + (a \# 1 + b \# g + c \# x + d \# gx)(\varepsilon\#g) \\
    &= \varepsilon \left[(g(b)+b)\#1 + (g(a)+a)\#g + (g(d)-d)\#x + (g(c)-c)\#gx\right].
\end{align*}
Thus, $a,b \in A^-$ and $c,d \in A^+$. Similarly, $a',b' \in A^+$ and $c',d' \in A^-$.
\end{proof}

Next we will show that the automorphisms are unipotent, that is,
the coefficients have no constant terms.

\begin{lemma}
Let $\phi \in \rAut(S)$. Then $\phi(u\#1)_0=\phi(v\#1)_0=0$.
\end{lemma}

\begin{proof}
By Lemma \ref{lem.wtsp}, $a,b,c',d' \in A^-$. However, any constant is in the positive
weight space and so we conclude that $a_0=b_0=c'_0=d'_0=0$.
%\begin{align*}
%\phi(u\#1)_0 &= c_0 \# x + d_0 \# gx \\
%\phi(v\#1)_0 &= a'_0 \# 1 + b'_0 \# g.
%\end{align*}
Now $\phi(v^2\#1)_0 = ((a')_0^2+(b')_0^2)\# 1 + 2(a')_0(b')_0 \# g$.
Since $\phi(v^2\#1) \in \kk[u^2,v^2]$, then $(a')_0=0$ or $(b')_0=0$. 

Next we look at the degree 1 component. Recall that $\phi(v^2\#1)_1=0$.
\begin{align*}
\phi(v^2\#1)_1 
%	&= [((a_0' + a_1') \# 1 + (b_0' + b_1') \# g + c_1' \# x + d_1' \# gx)^2]_1 \\
    &= (2a_0'a_1' + 2b_0'b_1') \# 1
    	+ (2a_0'b_1' + 2b_0'a_1') \# g
        + (2a_0'c_1' - 2b_0'd_1') \#x 
        + (2a_0'd_1' - 2b_0'c_1') \# gx.  
\end{align*}
If $a'_0\neq 0$, then $b'_0=0$ and this forces 
$a_1'=b_1'=c_1'=d_1'=0$ contradicting the surjectivity of $\phi$.
A similar argument holds if $b'_0 \neq 0$.
Thus, $a_0'=b_0'=0$. Furthermore,
\[ 0 = \phi((v\#1)(1\#x)-(1\#x)(v\#1)+(u\#1))_0 = c_0\#x + (2\tornado b_0' + d_0)\# gx.\]
Therefore, $c_0=0$ and $d_0=-2\tornado b'_0=0$.
\end{proof}

Now we show that $a$ and $b$ cannot have higher degree components. 

\begin{lemma}
\label{lem.abhd}
Let $\phi \in \rAut(S)$ and write $\phi(u\#1)$ as above. Then $a_k=b_k=0$ for $k>1$.
\end{lemma}

\begin{proof}
We have
\[ 0 = \phi((u\#1)(1\#x)+(1\#x)(u\#1)) = \tornado \left(x(a)\#1 + x(b)\#g + x(c)\#x + (2b+x(d))\#gx\right). \]
%\begin{align*}
%0 	&= \phi((u\#1)(x\#1)+(x\#1)(u\#1)) \\
%	&= (a \# x + b \# gx) + ( x(a)\#1 + x(b)\#g + x(c)\#x + x(d)\#gx + g(a)\#x - g(b)\#gx) \\
%    &= x(a)\#1 + x(b)\#g + x(c)\#x + (2b+x(d))\#gx.
%\end{align*}
Thus, $x(a)=x(b)=x(c)=0$ and $x(d) = -2b$.
Combining this with our computations above we have
\begin{align}
\label{eq.u2}
\phi(u^2\#1) = (a^2-b^2)\# 1 + (ab-ba)\#g
	+ (ac-ca + bd+db)\#x + (ad+da+bc-cb)\#gx.
\end{align}
Since $x(a)=x(b)=0$, then the $v$-degrees of all monomial summands
in both $a$ and $b$ are even. Thus, these monomials commute with one another.

%\textbf{Calling the degrees $d$ and $d'$ is $d'$angerous.}
%That's a really bad joke.
%I willingly concede that point.
%You do realize this will stay in the comments in the final draft?

Write $a = a_1 + a_2 + \cdots + a_m$ and $b = b_1 + b_2 + \cdots + b_{m'}$ where $a_i,b_i \in A_i$. 
Assume $\max\{m,m'\} > 1$. Since $(a^2-b^2)_{> 2}=0$ it follows that $m=m'$. Then
\[ 0 = (a^2-b^2)_{2m} = (a_m)^2-(b_m)^2. \]
Thus, $a_m = \pm b_m$. 
Consider the case $a_m=b_m$ and assume inductively that for some $\ell$,
$1 \leq \ell \leq d$, $a_k=b_k$ for all $\ell < k \leq d$. 
Then $a_ia_j - b_ib_j = 0$ for all $\ell < i,j < n$.
\[ 0 = (a^2-b^2)_{n+\ell} 
	= \sum_{i+j=\ell} a_ia_j - b_ib_j
    = 2a_na_\ell - 2b_nb_\ell \sum_{\substack{\ell < i,j < n \\ i+j=\ell}} a_ia_j - b_ib_j
    = 2a_n(a_\ell - b_\ell).
\]
Thus, $a_\ell = b_\ell$.
It follows from induction that $a_1=b_1$.
A similar proof in the negative case shows that $a_1=-b_1$.
But this contradicts $(a^2-b^2)_2 = (a_1^2-b_1^2) = \alpha u^2$ for some $\alpha \in \kk^\times$.
Therefore, $m,m' \leq 1$.
\end{proof}

We next determine the affine restricted automorphisms of $S$.
By the grading, this is equivalent to computing the linear parts of {\it any} restricted automorphism.

\begin{lemma}
\label{lem.types}
Suppose $\phi \in \rAut(S)$ is affine.
Let $\phi(1 \# g) = \varepsilon \# g$ and  $\phi(1 \# x) = \tornado \# x$
for some $\varepsilon = \pm 1$, $\tornado \in \kk^\times$.
Then $\phi( u \# 1)$ and $\phi(v\#1)$ take one of the two forms below
with $\alpha, \beta \in \kk$, $\alpha \neq 0$,
\begin{align*}
&\underline{\text{Type I}} ~~				
	&	&\underline{\text{Type II}}	 \\
&\phi(u\#1) = \alpha \left( u \# 1 \right) 	
	&	&\psi(u\#1) = \alpha \left( u\#g - 2 v \# gx\right) \\
&\phi(v\#1) = \alpha \left( \tornado \inv v \# 1 + \beta u \# x \right) 	
	&	&\psi(v\#1) = \alpha \left( \tornado \inv v\#g + \beta u \# gx \right).
\end{align*}
\end{lemma}

\begin{proof}
By Hypothesis \ref{hyp.auto} and Lemma \ref{lem.wtsp},
\begin{align*}
\phi(u\#1) &=\alpha_1 u\#1 + \alpha_2 u\#g + \alpha_3 v\#x + \alpha_4 v \#gx \\
\phi(v\#1) &=\beta_1 v\#1 + \beta_2 v\#g + \beta_3 u\#x + \beta_4 u\#gx.
\end{align*}
We have
\begin{align*}
0 	&= \phi((v\#1)(1\#x)-(1\#x)(v\#1)+(u\#1)) \\
	&= (\alpha_1-\tornado \beta_1)u\#1 + (\alpha_2-\tornado \beta_2)u\#g
	+ \alpha_3 v\#x + (\alpha_4+2\tornado\beta_2)v\#gx.
\end{align*}
Thus, $\alpha_3=0$, $\alpha_i=\tornado\beta_i$, $i=1,2$,
and $\alpha_4=-2\tornado\beta_2$. Furthermore,
\[ \phi((v\#1)^2)
= \left(\beta_1\beta_3 - \beta_2\beta_4)u^2
	+ (\beta_1^2+\beta_2^2)v^2\right)\#1
	+ \left(\beta_2\beta_3-\beta_1\beta_4)u^2 
    + 2\beta_1\beta_2v^2\right)\#g
\in \kk[u^2,v^2].
\]
Thus, $\beta_1=0$ or $\beta_2=0$.

If $\beta_2=\alpha_2=0$, then $\alpha_4=0$ by above.
Then $\alpha_1,\beta_1 \neq 0$, so $\beta_4=0$.

If $\beta_1=\alpha_1=0$, then $\phi((u\#1)^2) = (\tornado \beta_2)^2 u^2$,
so $\beta_2\neq 0$ forcing $\beta_3=0$.
\end{proof}

We say an automorphism $\phi \in \rAut(S)$ is of type I (resp. type II)
if its linear part is of type I (resp. type II) in Lemma \ref{lem.types}.
Observe that the composition of two automorphisms of type I or two automorphisms of type II yields an automorphism of type I, while the composition of a type I automorphism with a type II automorphism is of type II. We will show that if $\phi \in \rAut(S)$ is of type I (resp. type II), then it is even (resp. odd).
%\textbf{Is Type II surjective? Okay I answered my own question as $\psi(u\#g + 2v\#gx) = u\#1$. 
%Worth saving for later anyway.}
%Note that this automorphism group $G$ is a $\ZZ_2$-graded group with 
%$G_0$ the Type I automorphisms and $G_1$ the Type II automorphisms.

\begin{lemma}
If $\phi \in \rAut(S)$ is of type I, then $a=\alpha u$ and $b=d=0$.
Moreover, $b',c \in \cnt(S)$ and $a' = \alpha \tornado\inv v - c/2$.
\end{lemma}

\begin{proof}
The statement regarding $a$ and $b$ follows from Lemma \ref{lem.abhd}.
The coefficient in \eqref{eq.u2} is $0=ad+da=2ad$. Since $a\neq 0$, then $d=0$.

For the remainder, observe that if $r \in A^+$ then the $u$-degree of each summand is even.
If furthermore $x(r)=0$, then the $v$-degree of each summand is even. 
Consequently, if $r \in A^+$ and $x(r)=0$, then $r \in \kk[u^2,v^2]=\cnt(A)$.
Thus $c \in \cnt(S)$.

We have
\begin{align*}
0 	&= \phi((v\#1)(1\#x)-(1\#x)(v\#1)+(u\#1)) \\
	&= \tornado (a' \# x + b' \# gx) - \tornado ( x(a')\#1 + x(b')\#g + x(c')\#x + x(d')\#gx + g(a')\#x - g(b')\#gx) \\
 		&\quad + (a \# 1 + b \# g + c \# x + d \# gx) \\
  	&= (a-\tornado x(a'))\#1 + (b-\tornado x(b'))\#g + (c-\tornado x(c'))\#x + (d+2\tornado b'- \tornado x(d'))\#gx.
\end{align*}
Thus, $x(a')=\tornado\inv a$, $x(b')= \tornado \inv b$, $x(c')= \tornado \inv c$ and $x(d')=2b'+ \tornado \inv d$.

Since $b_1'=0$ and $x(b')=\tornado\inv b=0$ and so $b' \in \cnt(S)$.
%Similarly, $x(d)=-2 b=0$ and so $d \in \cnt(S)$.
%I think this line is not necessary...
As $x(a')= \tornado\inv a$ and $a_k=0$ for $k>1$, then $a' = \alpha \tornado \inv v + z$ for some $z \in \cnt(S)$.

%We'll need this next.
%\begin{align}
%\notag \phi((u\#1)(v\#1) &+ (v\#1)(u\#1)) \\
%	&= (a \# 1 + b \# g + c \# x + d \# gx)(a' \# 1 + b' \# g + c' \# x + d' \# gx)
%    	+ (a' \# 1 + b' \# g + c' \# x + d' \# gx)(a \# 1 + b \# g + c \# x + d \# gx) \\
%\notag	&= (aa' + bg(b') + cx(a') + dgx(b') + a'a + b'g(b) + c'x(a) + d'gx(b) ) \#1 \\
%\notag  &\quad + (ab' + bg(a') + cx(b') + dgx(a') + a'b + b'g(a) + c'x(b) + d'gx(a) ) \# g \\
%\notag  &\quad + (ac' + bg(d') + cg(a') + cx(c') - db' + dgx(d') \\
%\notag 	&\qquad + a'c + b'g(d) + c'g(a) + c'x(c) - d'b + d'gx(d)) \# x \\
%\notag 	&\quad + (ad' + bg(c') - cg(b') + cx(d') + da' + dgx(c') \\
%\notag 	&\qquad+ a'd + b'g(c) - c'g(b) + c'x(d) + d'a + d'gx(c)) \# gx \\
%\label{eq.uv1}	&= ((aa'+a'a) + (bb'-b'b) + ca - db) \#1 \\
%\label{eq.uv2}	&\quad + ( (ab'-b'a) + (ba'+a'b) + cb - da) \# g \\
%\label{eq.uv3}	&\quad + ( (ac'-c'a) + (d'b-bd') + (ca'+a'c) + (db'+b'd) + c^2+d^2) \# x \\
%\label{eq.uv4}	&\quad + ( (ad'+d'a) - (c'b+bc') + (b'c+cb') + (cd+dc) + (da' + a'd) ) \# gx 
%\end{align}

By above computations we have the following simplification.
\begin{align*}
\phi((u\#1)(v\#1) + (v\#1)(u\#1))
	&= ((aa'+a'a)+ ca) \#1 \\
	&\quad + ( (ac'-c'a) + 2ca' + c^2) \# x \\
    &\quad + ( (ad'+d'a) + 2b'c ) \# gx .
\end{align*}
From the identity component above we have 
\[ 0 	= aa' + a'a + ca
		= (\alpha u)(\alpha \tornado\inv v + z)+(\alpha \tornado\inv v + z)(\alpha u) + c(\alpha u)
    	= (\alpha u)(2z+c).\]
Since $\alpha \neq 0$, then $z = -c/2$.
\end{proof}

%Now we'll rewrite $\phi(v^2 \#1)$ in this case.
%\begin{align*}
%\phi(v^2 \#1)	
%	&= ( (a')^2 + (b')^2 + c'a)  \# 1 \\
%		&\quad + ( 2a'b' - d'a) \# g \\
%		&\quad + ( (a'c'+c'a') + c'c) \# x \\
%       &\quad + ( (a'd'+d'a') + d'c ) \# gx
%\end{align*}
%Thus, the $u$-degree of each summand $(a')^2+(b')^2$ is even.

%If $c'=0$, the $x$-component of $\phi((uv+vu) \#1)$ is
%\[ 0 = 2ca' + c^2 = c(2a'+c).\]
%However, since $c \in \cnt(S)$ and $a' \notin \cnt(S)$, then $c=0$.
%Now we assume $c=-2z$.

\begin{lemma} Let $\phi \in \rAut(S)$ be of type I. Then all monomials appearing as summands of $c'$ and $d'$ have even $v$-degree.
\label{lem.typei}
\end{lemma}
\begin{proof}
Based on the above computations, we have that
\begin{align*} \phi(u \# 1) &= \alpha u \# 1 + c \# x \\
\phi(v \# 1) &= \left(\alpha \tornado\inv v  - \frac{c}{2}\right) \# 1 + b' \# g + c' \#x  + d' \# gx
\end{align*}
for some $b', c \in \cnt (A)$, $c', d' \in A^-$. Since $\phi$ is an automorphism, there exists some $r \in A \# H$ such that $\phi(r) = v \# 1$. By using the relations in $A \# H$, we can write $r$ as a finite sum
\[ r = \sum_{i} \gamma_i \cdot (v\#1)^{i} (u\# 1)^{j_i} (1\#g)^{k_i} (1\#x)^{\ell_i}.
\]
for some $\gamma_i \in \kk$ and some $i, j_i, k_i, \ell_i \in \NN$. We therefore have
\[ v\# 1 = \sum_{i} \gamma_i \cdot \phi(v\#1)^{i} \phi(u\# 1)^{j_i} \phi(1\#g)^{k_i} \phi(1\#x)^{\ell_i}.
\]

Since $c \in \cnt(A) = \kk[u^2,v^2]$, for each $i$, we have
\[ \phi(u\# 1)^{j_i} \phi(1\#g)^{k_i} \phi(1\#x)^{\ell_i} = h_{i,1} \# 1 + h_{i,2} \#g + h_{i,3} \#x + h_{i,4} \# gx
\]
for some $h_{i,1}, h_{i,2}, h_{i,3}, h_{i,4} \in \kk[u,v^2]$. Therefore,
\begin{equation} \label{eq.v} v\# 1 = \sum_{i} \gamma_i \cdot \left( \left(\alpha \tornado\inv v  - \frac{c}{2}\right) \# 1 + b' \# g + c' \#x  + d' \# gx\right)^{i} \left(h_{i,1} \# 1 + h_{i,2} \#g + h_{i,3} \#x + h_{i,4} \# gx \right).
\end{equation}

We now consider the terms on the right-hand side that involve terms with odd $v$-degree. By Lemma~\ref{lem.autcnt}, $\phi(v\#1)^2 = (\kappa v^2 + z) \# 1$ for some $\kappa \in \kk^\times$, $z \in \kk[u^2]$, so these terms only occur when $i$ is odd. Suppose for contradiction that the largest odd power $N$ with a nonzero coefficient has $N \geq 3$. The largest odd $v$-degree appearing on the right-hand side then come from the term
\begin{align*} \gamma_N (\kappa v^2+z)^{\frac{N-1}{2}} \left( \left(\alpha\tornado\inv v  - \frac{c}{2}\right) \# 1 + b' \# g + c' \#x  + d' \# gx\right) \left(h_{N,1} \# 1 + h_{N,2} \#g + h_{N,3} \#x + h_{N,4} \# gx \right) .
\end{align*}
After multiplying (recalling that $h_{N,i} \in \kk[u,v^2]$), the identity component of this term is
\[\gamma_N (\kappa v^2+z)^{\frac{N-1}{2}} \left( \left(\alpha \tornado\inv v  - \frac{c}{2}\right)  h_{N,1} + b' g(h_{N,2}) \right) \# 1.
\] 
The product $\gamma_N (\alpha^2 \tornado^{-2} v^2+z) \left(\alpha \tornado\inv v  - \frac{c}{2}\right)  h_{N,1}$ involves the unique term of highest odd $v$-degree occurring in the identity component of \eqref{eq.v}. Since $N \geq 3$, this term must be zero. Hence, $h_{N,1} = 0$. By a similar computation in the other components, we conclude that $h_{N,2} = h_{N,3} = h_{N,4} = 0$. This is a contradiction.

Now there is a unique term in \eqref{eq.v} involving any terms with odd $v$-degree, which occurs when $i = 1$. Since $c' , d' \in A^-$, if $c'$ and $d'$ are non-zero, all of their monomial summands have odd $u$-degree. If any of their monomial summands have odd $v$-degree, this contradicts \eqref{eq.v}, as only one term involves any terms of odd $v$-degree, so the right-hand side must have a non-zero $x$ and $gx$ component with odd $v$-degree.
\end{proof}

We are now able to describe all automorphisms in $\rAut(S)$.

\begin{proof}[Proof of Theorem \ref{thm.raut}]
We have seen that any $\phi \in \rAut(S)$ is of type I or type II. 
Suppose $\phi$ is of type I. By Lemma~\ref{lem.typei}, each summand of $c'$ and $d'$ 
has even $v$-degree, and so $c = x(c') = 0$ and $b' = x(d')/2 = 0$. 
Since the $g$ component of $\phi(v^2 \# 1)$ is $(2a'b' - d'a) \# g$, 
we also conclude that $d' = 0$. Hence, $\phi$ is of the form
\begin{align*} &\phi(u \# 1) = \alpha u \# 1 &
& \phi(v \# 1) = \alpha \tornado\inv v \# 1 + c' \#x  \\
& \phi(1 \# g) = \varepsilon \# g & &\phi(1 \#x) = \tornado \# x
\end{align*}
for some $c' \in A^-$ with even $v$-degree. 
Now $\phi(v^2 \#1) = \alpha^2 \tornado^{-2} v^2 + c'u \# 1$, so by Lemma~\ref{lem.autcnt}, $c'$ is a polynomial in $u$ of odd degree and hence $\phi$ is of even type.

Now suppose that $\phi$ is of type II. Consider the following map:
\begin{align*}
	&\psi(u\# 1) = u\#g - 2v\# gx &    &\psi(v\# 1) = v\# g\\
& \psi(1 \# g) = 1 \# g & &\psi(1 \#x) = 1 \# x.
\end{align*}
It is clear that $\psi$ is an automorphism of type II.
Further,
\[ \phi(\psi(v\#1)) = (a'\# 1 + b'\# g + c'\#x + d'\#gx)(1\#g) = a'\# g + b'\# 1 - c'\#gx - d'\#x.\]
Since $\phi \circ \psi$ is an automorphism of type I,
then it is of even type.
Hence, $a'=c'=0$, $b'=\alpha \tornado\inv v$, 
and $d'$ is a linear combination of odd powers of $u$. Now
\begin{align*}
	\phi(\psi(u\#1)) 
    	&= (a\# 1 + b\# g + c\# x + d\# gx)(1\#g) - 2(b'\#g + d'\#gx)(1\#gx) \\
        &= (a\# g + b\# 1 - c\# gx - (d+2b')\#x.
\end{align*}
Thus, $a=c=0$, $b=\alpha u$, and $d=-2b' = -2\alpha \tornado\inv v$ and so $\phi$ is of odd type.
\end{proof}

\begin{question}What is the full automorphism group $\Aut(S)$?
\end{question}

\bibliographystyle{plain}

\end{document}